\documentclass{amsart}

\usepackage[OT4]{fontenc}
\usepackage[utf8]{inputenc} 
\usepackage{amssymb}
\usepackage{enumerate}
\usepackage{esvect}
\usepackage{mathtools}
\usepackage{bm}

\theoremstyle{plain}
\newtheorem{theorem}{Theorem}

\newtheorem{lemma}{Lemma}

\theoremstyle{definition}

\theoremstyle{remark}

\DeclareMathOperator{\conv}{conv}
\DeclareMathOperator{\card}{card}
\DeclareMathOperator{\Avg}{Avg}

\newcommand{\RR}{\mathbb{R}}

\DeclareMathOperator{\Vol}{Vol}

\begin{document}



\title[Hermite-Hadamard for simplices]{Another refinement of the right-hand side of the Hermite-Hadamard inequality for simplices}

\author[M. Nowicka]{Monika Nowicka}
\address{Institute of Mathematics and Physics, UTP University of Science and Technology, al. prof. Kaliskiego 7, 85-796 Bydgoszcz, Poland}
\email{monika.nowicka@utp.edu.pl}
\author[A. Witkowski]{Alfred Witkowski}
\email{alfred.witkowski@utp.edu.pl}
\date{16.12.2014}
\setlength{\parindent}{0pt}

\begin{abstract} 
In this paper, we establish a new refinement of the right-hand side of Hermite-Hadamard inequality for convex functions of several variables defined on simplices.
\end{abstract}

\maketitle

%




The classical Hermite-Hadamard inequality states that if $f:I\to\mathbb{R}$ is a convex function then for all $a<b\in I$ the inequality
$$f\left(\frac{a+b}{2}\right)\leq \frac{1}{b-a}\int_a^b f(t)dt\leq \frac{f(a)+f(b)}{2}$$
is valid. This powerful tool has found numerous applications and has been generalized in many directions (see e.g. \cite{DP} and \cite{Bes}). One of those directions is its multivariate version:
\begin{theorem}[\cite{Bes}]
	Let $f:U\to \mathbb{R}$ be a convex function defined on a convex set $U\subset \mathbb{R}^n$ and $\Delta\subset U$ be an $n$-dimensional simplex with vertices $x_0,x_1,\dots,x_n$, then
	\begin{equation}
	f(b_\Delta)\leq \frac{1}{\Vol \Delta}\int_\Delta f(x)dx\leq \frac{f(x_0)+\dots+f(x_n)}{n+1},
	\label{eq:HH for simplex}
	\end{equation}
	where $b_\Delta=\frac{x_0+\dots+x_n}{n+1}$ is the barycenter of $\Delta$ and the integration is with respect to the $n$-dimensional Lebesgue measure.
\end{theorem}
The aim of this note is to  proof a  refinement of the right-hand side  of \eqref{eq:HH for simplex}. 

Let us start with a set of definitions. 
	
A function $f\colon I\to\mathbb{R}$ defined on an interval $I$ is called \textit{convex} if for any $x,y\in I$ and $t\in(0,1)$ the inequality
$$f(tx+(1-t)y)\leq tf(x)+(1-t)f(y)$$
holds.

If $U$ is a convex subset of $\mathbb{R}^n$, then a function $f:U\to\mathbb{R}$ is convex if its restriction to every line segment in $U$ is convex.
 
For $n+1$ points $x_0,\dots, x_n\in\RR^n$ in general position the set  $\Delta=\conv\{x_0,\dots,x_n\}$ is called an $n$-dimensional \textit{simplex}. If $K$ is a nonempty subset of the set  $N=\{0,\dots,n\}$ of cardinality $k$, the set $\Delta_K=\conv\{ x_i: i\in K\}$ is called a \textit{face} (or a $k-1$-face) of $\Delta$. The point $b_K=\frac{1}{k}\sum_{i\in K}x_i$ is called a \textit{barycenter} of $\Delta_K$. The barycenter of $\Delta$ will be denoted by $b$. By $\card K$ we shall denote the cardinality of set $K$.

For each $k-1$-face $\Delta_K$ we calculate the average value of $f$ over $\Delta_K$ using the formula
$$\Avg(f,\Delta_K)=\frac{1}{\Vol(\Delta_K)}\int_{\Delta_K} f(x)dx,$$
where the integration is with respect to the $k-1$-dimensional Lebesgue measure (in case $k=1$ this is the counting measure). 

For $k=1,2,\dots,n+1$ we define 
$$\mathcal{A}(k)=\frac{1}{\binom{n+1}{k}} \sum_{\substack{{K\subset N}\\{\card K=k}}}\Avg(f,\Delta_K).$$

Note that the right-hand side of the inequality \eqref{eq:HH for simplex}  can be rewritten as $\mathcal{A}(n+1)\leq\mathcal{A}(1)$. 
It turns out, that 
\begin{theorem}\label{thm:main1}
			The following chain of inequalities holds:
	$$\mathcal{A}(n+1)\leq \mathcal{A}(n)\leq \dots\leq\mathcal{A}(2)\leq\mathcal{A}(1).$$
\end{theorem}

In the proof we shall use the following
\begin{lemma}[{\cite[Theorem 4.1]{NoWi1}}]\label{lem:WN1}
If $K_i=N\setminus\{i\}$ and $b$ is the barycenter of $\Delta$, then
$$\Avg(f,\Delta)\leq \frac{1}{n+1}f(b)+\frac{n}{n+1}\frac{1}{n+1}\sum_{i=0}^n \Avg(f,\Delta_{K_i}).$$
\end{lemma}
\begin{proof}[Proof of Theorem \ref{thm:main1}]
We shall prove first the inequality $\mathcal{A}(n+1)\leq \mathcal{A}(n)$. Let us use the notation from Lemma \ref{lem:WN1}. For $i=0,1,\dots,n$ we have
\begin{equation}
b_{K_i}=\frac{1}{n}\sum_{\substack{j=0\\j\neq i}}^n x_j=\frac{1}{n}\left(\sum_{j=0}^n x_j -x_i\right)=\frac{1}{n}((n+1)b-x_i).
\label{eq:b_is}
\end{equation}
Summing \eqref{eq:b_is} we obtain 
\begin{equation}
b=\frac{1}{n+1}\sum_{j=0}^n b_{K_j}.
\label{eq:b}
\end{equation}
Now using Lemma \ref{lem:WN1} and  convexity of $f$ applied to $\eqref{eq:b}$   we get
\begin{align*}
	\Avg(f,\Delta)&\leq \frac{1}{n+1}f(b)+\frac{n}{n+1}\frac{1}{n+1}\sum_{i=0}^n \Avg(f,\Delta_{K_i})	\\
	&	\leq \frac{1}{n+1}\frac{1}{n+1}\sum_{i=0}^n f(b_{K_i})+\frac{n}{n+1}\frac{1}{n+1}\sum_{i=0}^n \Avg(f,\Delta_{K_i})\\
	\intertext{thus, by the left-hand side of \eqref{eq:HH for simplex}}
	&\leq \frac{1}{n+1}\frac{1}{n+1}\sum_{i=0}^n \Avg(f,\Delta_{K_i})+\frac{n}{n+1}\frac{1}{n+1}\sum_{i=0}^n \Avg(f,\Delta_{K_i})\\
	&=\frac{1}{n+1}\sum_{i=0}^n\Avg(f,\Delta_{K_i}).
\end{align*}
This shows the inequality $\mathcal{A}(n+1)\leq \mathcal{A}(n)$. The other inequalities follow by simple induction argument applying the same reasoning to all terms in $\mathcal{A}(n)$ etc.
\end{proof}

Just for completeness note that similar refinement of the left-hand side of \eqref{eq:HH for simplex} can be found in \cite[Corollary 2.6]{NoWi2}. It reads as follows:
\begin{theorem}
	For  a nonempty subset $K$ of $N$ define the simplex $\Sigma_K$ as follows: let $A_K$ be the affine span of $\Delta_K$ and $A_K'$ be the affine space of the same dimension, parallel to $A_K$ and passing through the barycenter of $\Delta$. Then $\Sigma_K=\Delta\cap A_K'$.
	
	For $k=1,2,\dots,n+1$ we let 
$$\mathcal{B}(k)=\frac{1}{\binom{n+1}{k}} \sum_{\substack{{K\subset N}\\{\card K=k}}}\Avg(f,\Sigma_K).$$
Then
$$f(b)=\mathcal{B}(1)\leq \mathcal{B}(2)\leq\dots\leq \mathcal{B}(n+1)=\Avg(f,\Delta).$$
\end{theorem}


\begin{thebibliography}{9}
\bibitem{Bes}
Bessenyei, M: The Hermite–Hadamard inequality on simplices. American Mathematical Monthly \textbf{115}(4), 339–-345 (2008)
\bibitem{DP}
Dragomir, SS, Pearce, CEM: Selected topics on Hermite–Hadamard inequalities

\bibitem{NoWi1}
Nowicka, M  and  Witkowski, A: A refinement of the right-hand side of the Hermite-Hadamard inequality for simplices. , Aequat. Math., \textbf{91} (2017), 121-128, doi:10.1007/s00010-016-0433-z,
\bibitem{NoWi2}
Nowicka, M  and  Witkowski, A:  A refinement of the left-hand side of Hermite-Hadamard inequality for simplices. , J. Inequal. Appl., (2015), 2015:373, doi:10.1186/s13660-015-0904-0,
\end{thebibliography}
\end{document}